%% file: main.tex
\documentclass{paperclass}

\usepackage{enumitem}
\setitemize[1]{itemsep=5pt}
\setenumerate[1]{itemsep=5pt}

\definecolor{MK_One_Five}{RGB}{90,180,172}
\definecolor{MK_Two_Six}{RGB}{33,102,172}
\definecolor{MK_Three_Six}{RGB}{27,120,55}
\definecolor{MK_Three_One}{RGB}{118,42,131}
\hypersetup{linkcolor=MK_Three_Six,citecolor=MK_Two_Six,urlcolor=MK_Three_One}

\usepackage{lipsum}
\usepackage{parskip}

\numberwithin{equation}{section}

\makeatletter
\def\@setdate{Date: \today}
\makeatother

\title[Freeness of Pluricanonical Linear system]{Freeness of Pluricanonical Linear System on Smooth Minimal n-folds of General Type}
\author{Hanye Gu}
\address{\rm School of Mathematical Sciences, Fudan University, Shanghai 200433, China}
\email{24110180015@m.fudan.edu.cn}
\date{\today}

\begin{document}

\begin{abstract}
  Let $X$ be a smooth minimal projective variety of general type of dimension $n\geq 4$ over $\bb C$. We prove that $|\frac{n^2+n}{2}K_X|$ is base point free. This gives a new bound for the pluricanonical case of Fujita Freeness Conjecture for nef and big canonical divisors. 
\end{abstract}


\maketitle
\tableofcontents
\input{Introduction}

\input{Preliminary}
\input{MethodForSingularpoint}

\input{MainTheoremandApplication}

\printbibliography

\nocite{*}

\end{document}

%% file: Introduction.tex
\section{Introduction}
In this paper, we work over the field of complex numbers $\bb C$. Let $X$ be a smooth projective variety of dimension $n$ over $\bb C$.\par
Let $B$ be a Cartier divisor on $X$, the complete linear system 
$$|B|=\bb P (H^0(X,\ml O_X(B)))$$ 
gives rise to a rational map $\phi_{|B|}:X\dashrightarrow \bb P(H^0(X,B)^{\vee})$. If $|B|$ is base point free, then $\phi_{|B|}$ becomes a morphism to projective space, which is one of the main tools to study projective varieties. Further, if $X$ is a projective variety of general type, consider the canonical ring 
$$R(X,K_X)=\bigoplus_{m\geq 0}H^0(X,mK_X).$$
The canonical model of $X$ is defined to be $X_{can}:=\opn{Proj}R(X,K_X)$. So it is natural to study the freeness of adjoint and pluricanonical linear systems. Fujita, in~\cite{Fujita87}, posed the following conjecture:
\begin{conjecture}[\emph{Fujita Freeness Conjecture}]
\emph{Let $X$ be a smooth projective variety of dimension $n$ and $B$ be an ample Cartier divisor. Then $|K_X+mB|$ is base point free for every integer $m\geq n+1$.}  
\end{conjecture}
The bound for $m$ is sharp as one can take $X=\bb P^n$ and $B=H$ where $H$ is a hyperplane section.\par
Fujita Freeness conjecture has been proved in dimension at most $5$, see \cite{Rei88,EL93,Kawamata97,ye2017global,YZ20}. In higher dimensions, Angehrn and Siu~\cite{AS95} proved that $|K_X+mB|$ is base point free for every integer $m> \frac{n(n+1)}{2}$.  Koll\'ar~\cite{Kollar97} gave an algebraic version of their proof. Helmke~\cite{Hel97} developed an estimate of multiplicity and proved a quadratic bound. Later, Heier proved a bound of $O(n^{\frac{4}{3}})$ in \cite{Hei02}. Ghidelli and Lacini~\cite{GL24} refined Helmke's estimate and proved that $|K_X+mB|$ is base point free for $m\geq n(\log\log(n)+2.34)$. In a recent preprint, Han~\cite{Han26} first established a linear bound for arbitrary $n$. More precisely, he proved that $|K_X+mB|$ is base point free for every integer $m\geq 2n$.\par

There are also effective results when $B$ is merely nef and big. Ein and Lazarsfeld~\cite{EL93} proved that $|mK_X|$ is base point free for every integer $m\geq7$ where $X$ is a smooth minimal projective threefold of general type. Ma\c{s}ek, Ein and Lazarsfeld~\cite{ELM95} generalized the result to terminal $3$-folds. Matsushita~\cite{Matsushita96} further generalized their result to klt threefolds. Lee~\cite{Lee00} proved that $|4K_X|$ is base point free where $X$ is a projective canonical threefold with $K_X$ an ample Cartier divisor.\par

In higher dimensions, Koll\'ar~\cite{Kol93} established an effective base point free theorem for klt pair. Fujino, in~\cite{Fuj10}, further proved a local freeness criterion for lc pairs. More precisely, let $(X,\Delta)$ be lc and $M\equiv K_X+\Delta+N$ where $N$ is an ample divisor. Let $x\in X$ be a closed point.  If there are positive numbers $c(k)$ such that $N^{\opn{dim}Z}\cdot Z>[c(\opn{dim}Z)]^{\opn{dim}Z}$ for every positive-dimensional irreducible subvariety $Z\ni x$ and $\sum_{k=1}^{\opn{dim}X}k/c(k)\leq 1$, then $|M|$ is base point free at $x$. After passing to a canonical model, Fujino's theorem shows that $|\frac{n^2+n+4}{2}K_X|$ is base point free where $X$ is a smooth minimal projective variety of dimension $n$ of general type. 

In this paper, we restrict our interest to the freeness problem of  pluricanonical linear system when the canonical divisor is nef and big. Specifically, we consider the following question:
\begin{question}
\label{question2}
    Let $X$ be a smooth minimal projective variety of general type, i.e. $K_X$ is nef and big. Can we find an $m(n)$ such that $|mK_X|$ is base point free if $m\geq m(n)$?
\end{question}

In this paper, we prove the following theorem, which answers~\ref{question2}.
\begin{theorem}
\label{maintheorem3}
    Let $X$ be a smooth minimal projective variety of general type of dimension $n\geq 4$. Let $V=K_X^n$. If an integer $m$ satisfies $m\geq \max\{2n+1,R_n(V)+1\}$, where
    $$ R_n(V):=
 \begin{cases}
  \frac n{V^{1/n}}+1,&n=2,\\
  \frac n{V^{1/n}}+\frac{n(n-1)}{2}-1,&n\geq3,
 \end{cases}$$
    then $|mK_X|$ is base point free.
\end{theorem}
In particular, we have the following result.
\begin{theorem}
\label{maintheorem}
Let $X$ be a smooth minimal projective variety of general type of dimension $n\geq 4$. The pluricanonical linear system $|mK_X|$ is base point free for every integer $m\geq \frac{n(n+1)}{2}$.    
\end{theorem}

\subsection{Outline of Proof}
Before we outline the proof of our main theorems, we first recall the methof of cutting LC centers, which is the main framework.\par
First of all, our ideal case is the following: Let $X$ be a smooth projective variety of dimension $n$ and $B$ be a nef and big Cartier divisor. If there exist an effective $\bb Q$-Cartier divisor $G$ satisfying:
\begin{enumerate}
    \item $G\sim_{\bb Q}cB$ with $0<c<1$;
    \item $(X,G)$ is lc near $x$ but not klt at $x$;
    \item $\{x\}$ is the unique LC center through $x$.
\end{enumerate}
Then $x$ is an isolated point of the non-klt locus, and hence of the subscheme $T=V(\ml J(X,G))$. Since $B-G\sim_{\bb Q}(1-c)B$ is nef and big, Nadel vanishing gives $H^1(X,\ml O_X(K_X+B)\otimes \ml J(X,G))=0$. Consequently, we have the following surjective map:
$$H^0(X,\ml O_X(K_X+B))\twoheadrightarrow H^0(X,\ml O_X(K_X+B)|_T).$$
Since $x$ is an isolated point of $T$, the right hand side contains a section nonvanishing at $x$. We can lift it to $s\in H^0(X,\ml O_X(K_X+B))$ such that $s(x)\neq 0$, which proves the freeness of $|K_X+B|$ at $x$. For the detailed proof, see Proposition~\ref{prop:isolated-nadel}.\par
The remaining part is to construct such $G$, this is where the LC center cutting technique comes into play. We begin with an effective divisor $\Delta\sim_{\bb Q}cB$, such that $(X,\Delta)$ is lc near $x$ but not klt at $x$. Let $Z\ni x$ be the minimal LC center. If $\dim Z>0$, then we can construct an effective $\bb Q$-divisor $D$ such that $\opn{ord}_x (D|_Z)$ is sufficiently high and $D\sim_{\bb Q}B$. Under some numerical conditions, the pair $(X,\Delta+tD)$ where $t=\opn{lct}_x(X,\Delta;D)$ has a smaller minimal LC center $Z^\p\subsetneq Z$ and $\dim Z^\p<\dim Z$. The local discrepancy estimate controls the numerical cost of the operation, allowing us to repeat this procedure until $Z=\{x\}$. Once the minimal LC center is $\{x\}$, we can perturb it and make it the unique LC center through $x$, bringing us back to the ideal case.\par

Let us outline the proof of Theorem~\ref{maintheorem3}. Combine Theorem~\cite[Theorem~5.4]{Han26} with the canonical contraction $f:X\rightarrow Y$ in Proposition~\ref{prop:canonical contraction}, we prove that if $X$ is a smooth projective variety with $K_X$ nef and big, then $\opn{Bs}|mK_X|\subset \mathbf{B}_+(K_X)=\opn{Exc}(f)$ for every integer $m\geq 2n+1$.\par
Let $X$ be a smooth minimal projective variety of general type of dimension $n$. The second part is to prove $\opn{Bs}|mK_X|\cap \opn{Exc}(f)=\varnothing$ for sufficiently large $m$. We first pass to the canonical model $Y$. Note that $Y$ may have singularities, so we have to modify the LC-center-cutting technique to the singular case by normalized volume and adjunction in~\cite{fujino23}. With Han's new estimate about the local discrepancy and $\opn{mult}_xZ$, we can improve Fujino's bound by $2$. Combining all these together, we prove Theorem~\ref{maintheorem3}, from which Theorem~\ref{maintheorem} follows immediately.

\vskip 0.5cm

This paper is organized as follows: Section~2 collects the preliminaries. Section~3 deals with the canonical contraction and pluricanonical linear systems. Section~4 contains the proof of the main theorems and several applications.

\textit{Acknowledgments.} The author would like to express his gratitude to his advisor, Professor Meng Chen, for his patient guidance and constant support throughout his work. Further, the author wants to thank Professor Chen Jiang and Jingjun Han for their helpful discussion. 
The author is also grateful to his fellow students, including Hexu Liu, Minzhe Zhu, Wentao Chang, Mengchu Li, Sicheng Ding, Tianyue Zhang, Pinxian Bie, Peien Du, Zhengjie Yu, Pengjin Wang and Caoyang Zhu for their helpful discussion, encouragement and friendship. The author thanks ChatGPT 6 for identifying typographical and grammatical errors.

%% file: Preliminary.tex
\section{Preliminary}

\label{sec:preliminaries}

\subsection{Divisors, linear systems, and positivity}

\begin{definition}
Let $X$ be a normal variety.  A Weil divisor $D$ is \emph{Cartier} if it is
locally defined by a nonzero rational function.  It is \emph{$\bb{Q}$-Cartier} if
$mD$ is Cartier for some integer $m>0$. A $\bb{Q}$-divisor is a finite
$\bb{Q}$-linear combination of integral divisors.  For $\bb{Q}$-Cartier
$\bb{Q}$-divisors, $D_1\sim_{\bb{Q}}D_2$ means that $mD_1$ and $mD_2$ are integral Cartier divisors and $mD_1\sim mD_2$ for some $m>0$.\par
A \emph{prime divisor over $X$} is a prime divisor $E$ on a normal variety
$Y$ with proper birational morphism $Y\to X$, the center of $E$ on $X$ is
denoted by $c_X(E)$. 
\end{definition}\par

\begin{definition}
Let $D$ be a Cartier divisor and $|D|$ be the complete linear system. The base locus of $|D|$ is the closed set
$$
 \operatorname{Bs}|D|:=\{x\in X\,|\, s(x)=0\text{ for every }s\in H^0(X,\ml{O}_X(D))\}.
$$
\end{definition}

\begin{definition}
    The linear system $|D|$ is said to be base point free at $x$ if the evaluation map
$$H^0(X,\ml{O}_X(D))\to\ml{O}_X(D)\otimes k(x)$$
is surjective.  It is base point free if this holds at every point.
\end{definition}

\begin{definition}
Let $X$ be a normal variety and $B\in N^1_{\bb R}(X)$. $B$ is \emph{nef} if $B\cdot C\geq0$ for every
irreducible curve $C\subset X$.  It is \emph{big} if its numerical class lies
in the interior of the pseudo-effective cone. 
\end{definition}

\begin{definition}
     For a big $\bb{Q}$-Cartier
divisor $D$, the \emph{augmented base locus} is
$$
 \mathbf{B_+}(D):=\bigcap_{D\sim_{\bb{Q}}A+F}\operatorname{Supp} F,
$$
where the intersection is taken over all Kodaira decompositions with $A$ ample $\bb{Q}$-Cartier and $F\geq0$. 
\end{definition}
\begin{remark}
   $x\notin\mathbf{B_+}(D)$ if and only if there exist a Kodaira decomposition $D\sim_{\bb{Q}}A+F$ with $x\notin\operatorname{Supp} F$.
\end{remark}

\vskip 0.5cm

\subsection{Pairs and discrepancy}\par

\begin{definition}
A \emph{subpair} $(X,\Delta)$ consists of a normal variety $X$ and an $\bb{R}$-divisor $\Delta$ such that $K_X+\Delta$ is $\bb{R}$-Cartier. It is a \emph{pair} if $\Delta\geq 0$.  A boundary is
an effective $\bb{R}$-divisor whose coefficients belong to $[0,1]$. In this paper, all pairs are assumed to be $\bb Q$-pair otherwise stated.
\end{definition}\par
\begin{definition}
By  a \emph{log resolution} of a pair $(X,\Delta)$ and ideal sheaf $\ml I$, we mean a projective
birational morphism $f:Y\to X$ such that 
\begin{enumerate}
    \item $Y$ is smooth;
    \item $\ml I\ml O_Y=\ml O_Y(-F)$ for some effective Cartier divisor $F$;
    \item $\opn{Supp}(f_*^{-1}\Delta)\cup\opn{Exc}(f)\cup\opn{Supp}F$ is simple normal crossing.
\end{enumerate}
If $E$ is a prime divisor on a birational model $f:Y\rightarrow X$, write
$$
 K_Y+\Delta_Y=f^*(K_X+\Delta).
$$
The \emph{log discrepancy} of a prime divisor $E$ over $X$ is
$$
 a(E,X,\Delta):=1-\operatorname{coeff}_E(\Delta_Y).
$$
This number is independent of the birational model $Y$.

\end{definition}
\begin{remark}
    By "a log resolution that principalizes the sheaf $\ml I$", we mean a log resolution $f:Y\rightarrow X$ such that $\ml I\ml O_Y=\ml O_Y(-F)$ for some effective Cartier divisor $F$.
\end{remark}

\begin{definition}
The subpair $(X,\Delta)$ is sub-klt (resp. sub-lc) if for every prime divisor $E$ over $X$, $a(E,X,\Delta)>0$ (resp. $\geq 0$). If $\Delta\geq 0$, we remove the "sub" and get the usual klt and lc.
A pair is plt if it is lc and $a(E,X,\Delta)>0$ for every prime divisor exceptional over $X$.\par
$X$ is canonical (resp. terminal) if $K_X$ is $\bb Q$-Cartier and  $a(E,X,0)\geq1$ (resp. $>1$)for every exceptional prime divisor $E$ over $X$.
In particular, canonical and terminal varieties are klt with $\Delta=0$.\par
We say that $(X,\Delta)$ is lc (resp. klt) at $x$ if it is lc (resp. klt) in an open neighborhood of $x$.
\end{definition}

\begin{definition}
Assume $(X,\Delta)$ is lc.  An \emph{LC place} is a prime divisor $E$ over
$X$ such that $a(E,X,\Delta)=0$.  Its center $c_X(E)$ is called \emph{LC center}.
Different LC places may have the same LC center.\par
The \emph{non-klt locus} is the closed set defined by:
$$
 \operatorname{Nklt}(X,\Delta):=\{y\in X:(X,\Delta)\text{ is not klt at }y\}.
$$
Set-theoretically it is the union of all LC centers.
\end{definition}

The following Proposition indicates the existence of minimal LC center.
\begin{proposition}[{\cite[Proposition~1.5]{Kawamata97}}]
\label{thm:minimal-center}
Let $X$ be a normal variety and $\Delta$ an effective $\bb Q$-Cartier divisor such that $K_X+\Delta$ is $\bb Q$-Cartier. Assume that $X$ is klt and $(X,\Delta)$ is lc.  If $W_1$ and $W_2$ are LC
centers of $(X,\Delta)$, then every irreducible component of $W_1\cap W_2$ is also an LC center of $(X,\Delta)$.\par
In particular, if $(X,\Delta)$ is not klt at $x$, then there exist a unique minimal LC center through $x$ under inclusion. Moreover, the minimal LC center through $x$ is normal near $x$. 
\end{proposition}\par

\vskip 0.5cm

\begin{definition}[LC threshold]
Let $(X,\Delta)$ be lc in a neighborhood of $x$ and let $D\geq0$ be an
effective $\bb{Q}$-Cartier $\bb{Q}$-divisor.  Define the local LC threshold of $D$ by
$$
 \operatorname{lct}_x(X,\Delta;D)
 :=\sup\{t\geq0:(X,\Delta+tD)\text{ is lc in a neighborhood of }x\}.
$$
We denote $\operatorname{lct}_x(X,0;D)$ by $\operatorname{lct}_x(X;D)$.\par

\end{definition}

\subsection{Multiplier ideals and Nadel vanishing}

\begin{definition}[Multiplier ideal]
Let $X$ be a normal variety and let $G\geq0$ be an effective $\bb{Q}$-divisor such that $K_X+G$ is $\bb Q$-Cartier.  For a log
resolution $f:Y\to X$ of $(X,G)$, define
$$
 \mathcal{J}(X,G):=f_*\ml{O}_Y\bigl( \lceil K_Y-f^*(K_X+G)\rceil \bigr).
$$
The ideal is independent of the resolution.  The closed subscheme defined by $\mathcal{J}(X,G)$ is called multiplier subscheme, and its support is $\operatorname{Nklt}(X,G)$.
\end{definition}

\begin{theorem}[{\cite[Theorem~2.16]{Kollar97}}]
\label{thm:nadel}
Let $X$ be a normal projective variety and $N$ a line bundle on $X$. Assume that $N\equiv K_X+\Delta+M$ where $M$ is a nef and big $\bb Q$-Cartier $\bb Q$-divisor on $X$ and $\Delta$ is an effective $\bb Q$-divisor such that $K_X+\Delta$ is $\bb Q$-Cartier. Then $ H^i\bigl(X,\ml{O}_X(N)\otimes\mathcal{J}(X,\Delta)\bigr)=0$ for all $i>0$ and $\opn{Supp}(\ml O_X/\ml J(X,\Delta))=\opn{Nklt}(X,\Delta)$.
\end{theorem}

\vskip 0.5cm

\begin{definition}
Let $Z$ be an irreducible variety of dimension $d$ and $x\in Z$ be a closed point.  The \emph{multiplicity} of $Z$ at $x$ is the Hilbert-Samuel multiplicity determined by
$$
 \operatorname{length}\bigl(\ml{O}_{Z,x}/\mathfrak m_x^k\bigr)
 =\frac{\operatorname{mult}_xZ}{d!}k^d+O(k^{d-1}).
$$\par

\end{definition}

\begin{definition}
    Let $X$ be an integral variety, $x\in X$ be a closed point. For $0\neq g\in\ml O_{X,x}$, we define the $\mfk{m}_x$-adic order to be
    $$\nu_{\mfk{m}_x}(g):=\max\{s\in \bb Z_{\geq 0}\,|\,g\in \mfk{m}_x^s\}.$$
    Let $D\geq 0$ be an effective $\bb Q$-Cartier divisor on $X$ and $m$ be a positive integer such that $mD$ is an effective Cartier divisor near $x$, and let $f\in \ml{O}_{X,x}-\{0\}$ be the local defining equation of $mD$. Define the order of $D$ at $x$ to be
    $$\opn{ord}_x D:=\frac{1}{m}\lim\limits_{k\rightarrow +\infty}\frac{\nu_{\mfk{m}_x}(f^k)}{k}.$$
\end{definition}
\begin{remark}
    When $x$ is a smooth point, then 
    $$\opn{ord}_x D=\frac{1}{m}\max\{s\in\bb Z_{\geq 0}\,|\,f\in\mfk{m}_x^s\}.$$
\end{remark}
\vskip 0.5cm

\begin{theorem}[{\cite[Theorem~4.1]{Kollar97}}]
\label{thm:bertini}
    Let $X$ be a smooth variety over a field of characteristic zero and $|B_1,\cdots, B_k|$ the linear system spanned by the effective divisors $B_i$. Let $B\in |B_1,\cdots, B_k|$ be a general member. Then for every $x\in X$, we have
    $$\opn{mult}_x B\leq 1 +\inf_{i}\{\opn{mult}_x B_i\}.$$
\end{theorem}
\begin{remark}
    This is also known as Bertini Theorem. The classical Bertini Theorem says that on a smooth variety, a general member of a base point free linear system is again smooth. Kollar's version also covers the classical version.
\end{remark}

\vskip 0.5cm

\subsection{Local discrepancy}\par
\begin{definition}
    Let $X$ be a projective variety, $x\in X$ be a closed point. Let $R=\ml O_{X,x}$ be the local ring at $x$ and $\mfk{m}_x$ be the maximal ideal of $R$ at $x$. Let $\mfk{a}\subsetneq R$ be an ideal, we say $\mfk{a}$ is $\mfk{m}_x$-\emph{primary} if $\sqrt{\mfk{a}}=\mfk{m}_x$ where 
    $$\sqrt{\mfk{a}}=\{h\in R\,|\,h^r\in \mfk{a}\text{ for some }r\geq 1\}.$$
\end{definition}

\begin{definition}
    Let $X$ be a normal variety and $x\in X$ be a closed point. Let $\mfk{a}\subset \ml O_{X,x}$ be a nonzero ideal. Let $E$ be a prime divisor over $X$ that lives on a normal birational model $f:Y\rightarrow X$ such that $x\in c_X(E)$. If $\mfk{a}=(g_1,\cdots,g_s)$, then we define the order of $\mfk{a}$ at $E$ to be  
    $$\opn{ord}_E(\mfk{a}):=\min\limits_{1\leq i\leq s}\opn{ord}_E(g_i),$$
    where $\opn{ord}_E(g_i)$ is the coefficient of $E$ in the $\opn{div}(f^*g_i)$.
\end{definition}

\begin{definition}
 Let $X$ be a normal variety and $x\in X$ be a closed point. Let $\Delta\geq 0$ be a $\bb Q$-divisor such that $K_X+\Delta$ is $\bb Q$-Cartier and $(X,\Delta)$ is lc near $x$. For a nonzero ideal $\mfk{a}\subset \ml O_{X,x}$, we define the \emph{local LC threshold} by
 $$\opn{lct}_x(X,\Delta;\mfk{a}):=\inf\limits_{\underset{x\in c_X(P),\,\opn{ord}_P(\mfk{a})>0}{P\text{ prime over }X}}\frac{a(P,X, \Delta)}{\opn{ord}_P(\mfk{a})}.$$
 If the index set is empty, then we define $\opn{lct}_x(X,\Delta;\mfk{a})=+\infty$.
\end{definition}
\begin{remark}
Let $\mfk{a}$ be an $\mfk{m}_x$-primary ideal, then $x\in c_X(E)$ and $\opn{ord}_E(\mfk{a})>0$ if and only if $c_X(E)=\{x\}$.
\end{remark}

\begin{remark}
\label{rmk:numerical description of ideal lct}
    For an $\mfk{m}_x$-primary ideal $\mfk{a}$, choose a log resolution $f:Y\rightarrow X$ over a neighborhood of $x$ such that $K_Y+\Delta_Y=f^*(K_X+\Delta)$, where $\Delta_Y=\sum d_iE_i$ and $F=\sum r_iE_i$, $\mfk{a}\ml O_Y=\ml O_Y(-F)$ and $\opn{Supp}\Delta_Y\cup \opn{Supp}F$ is SNC. Then 
    $$\opn{lct}_x(X,\Delta;\mfk{a})=\min_{r_i>0}\frac{1-d_i}{r_i}.$$\par
    Here, $\mfk{a}$ is interpreted as a coherent ideal sheaf $\ml I$ in a neighbourhood $U$ of $x$ such that $\ml I\subseteq\ml O_U$, $\ml I_x=\mfk{a}$ and $\opn{Supp}\ml O_U/\ml I=\{x\}$. Such $\ml I$ exist since $\mfk{a}$ is primary.
\end{remark}
The following definition is due to Helmke~\cite{Hel97} and Ein~\cite{Ein97}.
\begin{definition}[Local discrepancy]
    \label{prop:numerical local discrepancy}
    Let $(X,\Delta)$ be a $\bb Q$-pair and lc near a closed point $x\in X$. Let $\mfk{m}_x$ be the ideal sheaf of $x$. Then we define the \emph{local discrepancy} $b_x(X,\Delta)$ of $(X,\Delta)$ to be
    $$b_x(X,\Delta):=\opn{lct}_x(X,\Delta;\mfk{m}_x).$$
\end{definition}

\begin{remark}
   $b=b_x(X,\Delta)$ is a rational number. Moreover, the local discrepancy satisfies $0\leq b\leq n$ and $b=0$ if and only if $\{x\}$ is the minimal center of $(X,\Delta)$.
\end{remark}

\vskip 0.5cm

\begin{definition}
    Let $X$ be a normal projective variety of dimension $n$, $\mfk{a}$ be a $\mfk{m}_x$-primary ideal. Let $R=\ml O_{X,x}$ and $\mfk{m}_x$ to be the ideal sheaf of $x$. Define the \emph{Hilbert-Samuel multiplicity} to be the positive integer $e(\mfk{a})$ such that $$l_R(R/\mfk{a}^k)=\frac{e(\mfk{a})}{n!}k^n+O(k^{n-1}).$$
\end{definition}

\begin{definition}[{\cite[Theorem~27]{Liu18}}]
    Let $X$ be a normal projective variety of dimension $n$ and $x\in X$ be a closed point. Let $R=\ml O_{X,x}$ and $\mfk{m}_x$ is the maximal ideal at $x$. Assume that $X$ is klt near $x$. Define the \emph{normalized volume} to be 
    $$\widehat{\opn{vol}}(x,X)=\inf\limits_{\mfk{a}\,\mfk{m}_x\text{-primary}}\opn{lct}_x(X;\mfk{a})^n e(\mfk{a}).$$
\end{definition}

\begin{theorem}[{\cite[Theorem~1.6]{LX19}}]
\label{thm:bdd normal vol}
    Let $x\in X$ be an $n$-dimensional klt singularity. Then $\widehat{\opn{vol}}(x,X)\leq n^n$. The equality holds if and only if $X$ is smooth at $x$.
\end{theorem}

\subsection{Discriminant and moduli $b$-divisors}
This part contains the work of Fujino and Hashizume. For more details, see \cite{fujino23}.

    \begin{definition}[$b$-divisor]
    Let $X$ be a normal variety and let $\opn{Div}X$ be the space of Weil divisors on $X$. A \emph{b-divisor} on $X$ is an element $$\mathbf{D}\in\mathbf{Div}X=\lim\limits_{Y\rightarrow X}\opn{Div} Y,$$
    where the projective limit is taken over all proper birational morphism $f:Y\rightarrow X$ from a normal variety $Y$ under the pushforward homomorphism $f_*:\opn{Div}Y\rightarrow \opn{Div}X$. We can define $\bb Q$-$b$-divisors on $X$ similarly.\par
    If $\mathbf{D}=\sum d_i D_i$ is a $\bb Q$-$b$-divisor on normal variety $X$ and $f:Y\rightarrow X$ is a proper birational morpshim from a normal variety $Y$, then we define the \emph{trace} of $\mathbf{D}$ on $Y$ to be the $\mathbb{Q}$-divisor $$\mathbf{D}_Y:=\sum\limits_{D_i\text{ is a divisor on }Y}d_iD_i.$$
    The $\bb Q$ \emph{Cartier closure} of a $\bb Q$-Cartier $\bb Q$-divisor $D$ on a normal variety $X$ is the $\bb Q$-$b$-divisor $\overline{\mathbf{D}}$ with trace $(\overline{\mathbf{D}})_Y=f^*D$, where $f:Y\rightarrow X$ is a proper birational morphism from a normal variety $Y$.\par
    If $F$ is a prime divisor over $X$, we set $\opn{mult}_F \mathbf{D}:=\opn{mult}_F \mathbf{D}_{Y}$ where $Y$ is any birational model on which $F$ appears. We say that $\mathbf{D}$ \emph{descends} to $Y$ if $\mathbf{D}=\overline{\mathbf{D}_Y}$.
\end{definition}

\begin{definition}[Canonical-$b$-divisor]
    Let $X$ be a normal variety and let $\omega$ be a nonzero top rational differential form of $X$. Then $(\omega)$ defines a $b$-divisor $\mathbf{K}_X$. We call $\mathbf{K}_X$ the \emph{canonical b-divisor} of $X$.
\end{definition}

The following proposition is \cite[Proposition~2.7]{Han26}. For the proof and details, also see \cite[Theorem~1.2, Definition~1.3]{fujino23}
\begin{proposition}[{\cite[Proposition~2.7]{Han26}}]\label{thm:adjunction and inverse of adjunction}
    Let $(X,\Delta)$ be a $\bb Q$-pair and $W\subset X$ be an LC center of $(X,\Delta)$. Assume that $(X,\Delta)$ is lc over the generic point of $W$. Let $\nu: Z\rightarrow W$ be the normalization. Then there exist a \emph{discriminant b-divisor} $\mathbf{B}$ and a \emph{moduli b-divisor} $\mathbf{M}$ over $Z$ with the following properties.
    \begin{enumerate}
        \item The trace $\mathbf{B}_Z$ is effective and 
        $$\nu^*((K_X+\Delta)|_W)\sim_{\bb Q}K_Z+\mathbf{B}_Z+\mathbf{M}_Z.$$
        \item There exist a smooth quasi-projective variety $Z^\p$ and a projective birational morphism $p:Z^\p\rightarrow Z$ such that $\mathbf{K}_Z+\mathbf{B}=\overline{K_{Z^\p}+\mathbf{B}_{Z^\p}}$, $\mathbf{M}=\overline{\mathbf{M}_{Z^\p}}$, and 
        $$K_{Z^\p}+\mathbf{B}_{Z^\p}+\mathbf{M}_{Z^\p}\sim_{\bb Q}p^*\nu^*((K_X+\Delta)|_W).$$\par
        Moreover, if $Z$ is projective, then $Z^\p$ is projective and $\mathbf{M}_{Z^\p}$ is nef.
        \item Let $\widetilde{Z}\rightarrow Z$ be a birational model and let $P$ be a prime divisor $\widetilde{Z}$. For every LC place $T$ of $(X,\Delta)$ with center $W$, choose a log resolution $\pi:\widetilde{X}\rightarrow X$ on which $T$ appears and for which the induced rational map $T\dashrightarrow \widetilde{Z}$ is a morphism $f_T:T\rightarrow \widetilde{Z}$. If $K_T+\Delta_T=(K_{\widetilde{X}}+\Delta_{\widetilde{X}})|_T$, then
       $$1-\opn{mult}_P \mathbf{B}=\inf_{T}\sup\left\{ \lambda\in \bb R\,|\,(T,\Delta_T+\lambda f^*_TP)\text{ is sub-lc over }\eta_p\right\}.$$
    \end{enumerate}
\end{proposition}

%% file: MethodForSingularpoint.tex
\section{Canonical contraction and Pluricanonical Linear System}
In this section, we modify the LC-center-cutting technique to singular case.
Proposition~\ref{prop:canonical contraction} allows us to pass our question to canonical model $Y$. However, $Y$ maybe singular, so we need Proposition \ref{prop:initial singular pair}, Lemma~\ref{lem:singular-cut} and Lemma~\ref{lem:improved-finish} to cut LC centers in singular case. And Proposition~\ref{prop:improved-singular} is the final outcome of Section~3.\par
\vskip 0.5cm
We first introduce our ideal case.
\begin{proposition}
\label{prop:isolated-nadel}
Let $X$ be a normal projective variety. Let $D$ be a Cartier divisor on $X$. Let $G\geq0$ be an effective $\bb Q$-divisor such that $K_X+G$ is $\bb Q$-Cartier. Assume the following:
\begin{enumerate}
    \item  $(X,G)$ is lc near $x$;
    \item  $\{x\}$ is the unique lc center through $x$;
    \item  $D-(K_X+G)$ is nef and big. 
\end{enumerate}
Then there exist $s\in H^0(X,\ml O_X(D))$ such that $s(x)\neq 0$ 
\end{proposition}

\begin{proof}
Let $\ml{J}:=\ml{J}(X,G)$.  Consider a log resolution $f:Y\rightarrow X$. 
By the definition of multiplier sheaf, $\ml J(X,G)=f_*\ml{O}_Y\bigl( \lceil K_Y-f^*(K_X+G)\rceil \bigr)$. This is a coherent ideal sheaf in $\ml O_X$. We know that $(X,G)$ is klt near $x\in X$ if and only if $\ml J(X,G)_x=\ml O_{X,x}$. Consider the closed subscheme $T:=V(\ml J)$, the sheaf $\ml O_T=\ml O_X/\ml J$ has support $\opn{Nklt}(X,G)$.\par
By our assumption, $(X,G)$ is lc near $x$, the non-klt locus is the union of all LC centers intersecting with the neighbourhood. Moreover, since $\{x\}$ is the only LC center through $x$, we can shrink $U$ such that $\opn{Supp}T\cap U=\{x\}$. Since $x$ is an isolated point of $\opn{Supp}T$, we have 
$$
 \ml{O}_T\simeq \ml{O}_{T_x}\oplus\ml{O}_{T'}.
$$
Note that $T_x$ is zero-dimensional and Noetherian, hence is the spectrum of an Artin local ring $A=\ml O_{X,x}/\ml J_x$ and its residue field is $k(x)$.  Choose an open neighbourhood $V$ of $x$ such that the $\ml O_X(D)|_V\cong \ml O_V$. Restricting to $T_x$ gives $\ml O_X(D)|_{T_x}\cong \ml O_{T_x}$. The identity element $1\in A$ determines a section $e\in H^0(T_x,\ml O_X(D)|_{T_x})$. So $s:=(e,0)\in H^0(T,\ml O_X(D)|_T)$ satisfies $s(x)\neq 0$.\par
Next, we want to lift $s$ from $T$ to the whole $X$. We have the following exact sequence:
$$0\rightarrow \ml J\rightarrow \ml O_X\rightarrow \ml O_T\rightarrow 0.$$
Since $\ml O_X(D)$ is locally free, tensoring $\ml O_X(D)$ preserves exactness:
$$0\rightarrow \ml O_X(D)\otimes \ml J\rightarrow \ml O_X(D) \rightarrow \ml O_X(D)|_T\rightarrow 0.$$
So we have a cohomology sequence
$$H^0(X,\ml O_X(D))\rightarrow H^0(T,\ml O_X(D)|_T)\rightarrow H^1(X,\ml O_X(D)\otimes \ml J).$$\par
Since $D-K_X-G$ is nef and big, by Theorem~\ref{thm:nadel}, Nadel Vanishing gives $H^1(X,\ml{O}_X(D)\otimes\ml{J})=0$.
\end{proof}

\begin{proposition}
\label{prop:canonical contraction}
    Let $X$ be a smooth projective variety with $K_X$ nef and big. Then there exist a projective birational morphism $f:X\rightarrow Y$ such that:
    \begin{enumerate}
        \item $Y$ is a normal projective variety, $f_*\ml O_X=\ml O_Y$ and $f$ has connected fibers;
        \item $K_Y$ is an ample Cartier divisor and under a suitable choice of rational $n$-form, we have $K_X=f^*K_Y$;
        \item $Y$ is canonical and $f$ is isomorphic on $Y_{reg}$;
        \item $H^0(Y,mK_Y)\cong H^0(X,mK_X)$ and $\opn{Bs}|mK_X|=f^{-1}(\opn{Bs}|mK_Y|)$.
    \end{enumerate}
\end{proposition}
\begin{proof}
\emph{Step 1: Construct $Y$ and prove condition~(1).} Since $X$ is smooth, $(X,0)$ is klt and $K_X$ is Cartier. Apply \cite[Theorem~1.3]{Fuj09} to $L=K_X$ and $a=2$. Since $(X,0)$ is klt, it has no LC centers, so the log big is automatic. Then $\ml O_X(pK_X)$ is free for sufficiently large $p$. Therefore, we can choose a sufficiently large $p$ such that $pK_X$ and $(p+1)K_X$ are both free. Since $K_X$ is big and $pK_X$ is free, the map $\varphi_p:X\rightarrow \bb P^{N_p}$ is generically finite onto its image.\par
Consider the map $\Phi=(\varphi_p,\varphi_{p+1}):X\rightarrow \bb P^{N_p}\times \bb P^{N_{p+1}}$. Let $Z$ be its reduced image. Since $\varphi_p$ and $\varphi_{p+1}$ is generically finite, $\Phi$ is also generically finite.\par
Since $X$ is a projective variety and $Z$ is seperated, $\Phi$ is proper. So take Stein factorization
$$X\xrightarrow{f}Y\xrightarrow{g}Z,$$
where $g$ is finite, $f_*\ml O_X=\ml O_Y$ and $f$ is proper with connected fibers. Since $X$ is smooth and $\Phi$ is generically finite, $Y$ is the normalization of $Z$ in $K(X)$. Thus $Y$ is normal. Moreover, since $\Phi$ is proper and generically finite, we conclude that $K(Y)=K(X)$, which means $f$ is birational. Furthermore, $Y$ is a projective variety since $g$ is finite and $Z$ is projective. Given a closed immersion $i:X\hookrightarrow \bb P^r$, define $pr_Y\circ(i,f):X\rightarrow \bb P^r\times Y\rightarrow Y$. We can factorize $(i,f)$ as
$$X\xrightarrow{\gamma_f}X\times Y\xrightarrow{(i,\opn{id})}\bb P^r\times Y,$$ 
where $\gamma_f(x)=(x,f(x))$. Since $Y$ is seperated, the diagonal map is a closed immersion and the graph is obtained by pulling back the diagonal. So $\gamma_f$ is a closed immersion. The second map is the base change of the closed immersion $i$ by $Y\rightarrow \opn{Spec}\bb C$. So $f=pr_Y\circ (i,f)$ is projective.\par
\medskip
\emph{Step 2: Construct an intermediate line bundle and prove condition~(2).} For (2), we need to construct an ample line bundle $\ml L$ on $Y$ such that $f^*\ml L=\ml O_X(K_X)$. Consider $\ml A_p:=g^*(\ml O(1,0)|_Z)$ and $\ml A_{p+1}:= g^*(\ml O(0,1)|_Z)$. By construction, $f^*\ml A_p=\ml O_X(pK_X)$ and $f^*\ml A_{p+1}=\ml O_X((p+1)K_X)$. Therefore, the line bundle $\ml L=\ml A_{p+1}\otimes \ml A_p^{-1}$ satisfies $f^*\ml L=\ml O_X(K_X)$. By the projection formula, we have
$$\ml L^{\otimes (2p+1)}\cong \ml A_p\otimes\ml A_{p+1}=g^*(\ml O(1,1)|_Z).$$\par
Note that $\ml O(1,1)|_Z$ is ample, and its pullback by the finite morphism $g$ is also ample. Hence $\ml L$ is ample.\par
Next, we want to prove that $\ml L\cong\ml O_Y(K_Y)$ and $K_X=f^*K_Y$. Note that a proper birational morphism to a normal variety is isomorphic over an open subset $U$ satisfying $\opn{codim}_Y Y-U\geq 2$, so we have $f^{-1}(U)\cong U$ and $U\subset X$ is smooth. Then restricting to $f^{-1}(U)$ gives
$$\ml L|_{U}\cong \omega_U.$$
On the other hand, $\ml O_Y(K_Y)|_U\cong \omega_U$. So $\ml L|_U \cong \ml O_Y(K_Y)|_U$. Note that $\ml L$ and $\ml O_Y(K_Y)$ are both reflexive sheaves, so this isomorphism extends uniquely to the whole $Y$. In other words, we have $\ml O_Y(K_Y)\cong \ml L$. Therefore $K_Y$ is Cartier and ample.\par
Since $K(X)=K(Y)$, after a suitable choice of nonzero rational $n$-form $\omega$ , we may assume $K_X=\opn{div}_X(\omega)$ and $K_Y=\opn{div}_Y(\omega)$. Since $f$ is isomorphic in an open subset of $Y$ whose complement has codimension at least $2$, we have $f_*K_X=K_Y$. Define $D:=K_X-f^*K_Y$, then $D$ is a Cartier divisor and $D\sim 0$. Thus there exist a nonzero rational function $g\in K(X)^*$ such that $D=\opn{div}_X(g)$. We want to prove that $D=0$.\par
Pushing-forward through $f$, we have $f_*D=f_*K_X-f_*f^*K_Y=0$. And since $f$ is isomorphic in codim $1$, we have $f_*\opn{div}_X(g)=\opn{div}_Y(g)=0$. $Y$ is normal, so $g$ is also a unit in $H^0(Y,\ml O_Y)$. Then pullback by $f$, we have $D=\opn{div}_X(g)=0$. Therefore $K_X=f^*K_Y$. This proves (2).\par
\medskip
\emph{Step 3: Show that $Y$ has canonical singularity.} Let $P$ be an arbitrary prime divisor over $Y$. On a common smooth birational model $h:H\rightarrow X$ where $P$ appears, we have $K_H-(f\circ h)^*K_Y=K_H-h^*K_X$. So $a(P,Y,0)=a(P,X,0)$. Since $X$ is smooth, $K_H-h^*K_X$ is effective. Thus 
$$a(P,X,0)=1+\opn{coeff}_P(K_H-h^*K_X)\geq 1,$$
which means $Y$ is canonical.\par
\medskip
\emph{Step 4: Show that $f$ is isomorphic on $Y_{reg}$.} Set $U=Y_{reg}$, $V=f^{-1}(U)$ and $\varphi=f|_V:V\rightarrow U$. Both $V$ and $U$ are smooth. Note that the determinant of the differential is a section of $\omega_V\otimes \varphi^*\omega_U^{-1}$, whose zero divisor is $K_V-\varphi^*K_U=0$. Hence the determinant never vanishes. So $\varphi^*\Omega_U^1\rightarrow \Omega_V^1$ is an isomorphism and $\Omega_{V/U}^1=0$. Therefore, $\varphi$ is unramified, hence quasi-finite. Since $\varphi$ is proper, we conclude that $\varphi$ is finite. Finite birational morphism to normal variety is an isomorphism. So we proved (3).\par
\medskip
\emph{Step 5: Prove condition~(4).} Since $K_X=f^*K_Y$, we have $\ml O_X(mK_X)\cong f^*\ml O_Y(mK_Y)$ for every $m\in \bb Z_{>0}$. By projection formula, we have 
$$f_*\ml O_X(mK_X)\cong \ml O_Y(mK_Y)\otimes f_*\ml O_X\cong \ml O_Y(mK_Y).$$
Taking global section gives the isomorphism
$$f^*:H^0(Y,mK_Y)\rightarrow H^0(X,mK_X).$$\par
For any $x\in X$, let $y=f(x)$. Then by definition, $x\in \opn{Bs}|mK_X|$ if and only if $(f^*s)(x)=s(f(x))\otimes 1=0$ for every $s\in H^0(Y,mK_Y)$ if and only if $s(y)=0$ for every $s\in H^0(Y,mK_Y)$ if and only if $y\in \opn{Bs}|mK_Y|$. So we have $\opn{Bs}|mK_X|=f^{-1}(\opn{Bs}|mK_Y|)$.
\end{proof}

With the above proposition, we may pass our question to the canonical model of $X$. However, the canonical model maybe singular. Thus we need the following several propositions to perform LC-center-cutting technique on the canonical model.

\vskip 0.5cm

\begin{proposition}
\label{prop:initial singular pair}
    Let $X$ be a projective klt variety of dimension $n$. Let $A$ be an ample Cartier divisor and $x\in X$ be a closed point. Let $\beta>0$ satisty $\widehat{\opn{vol}}(x,X)<\beta^n A^n$. Then there exist an effective $\bb Q$-Cartier divisor $\Delta_0$ and $0<t_0<\beta$ such that $\Delta_0\sim_{\bb Q}t_0A$ and $(X,\Delta_0)$ is lc near $x$ but not klt at $x$.
\end{proposition}

\begin{proof}
    By the definition of normalized volume, there is an $\mfk{m}_x$-primary ideal $\mfk{a}$ such that $c^ne<\beta^n V$, where $c:=\opn{lct}_x(X;\mfk{a})$, $e:=e(\mfk{a})$ and $V=A^n$. Therefore, we can find a rational number $\rho$ such that $\rho >(\frac{e}{V})^{\frac{1}{n}}$ and $c<\frac{\beta}{\rho}$. For sufficiently large and sufficiently divisible $m=k\rho$, asymptotic Riemann-Roch gives
    $$h^0(X,\ml O_X(mA))=\frac{A^n}{n!}m^n+O(m^{n-1})=\frac{\rho^nA^n}{n!}k^n+O(k^{n-1}).$$
    On the other hand, the Hilbert-Samuel theorem gives
    $$l_R(R/\mfk{a}^k)=\frac{e}{n!}k^n+O(k^{n-1}).$$\par
    Since $\rho^nA^n>e$, we have $h^0(X,\ml O_X(mA))>l_R(R/\mfk{a}^k)$. So the kernel of map
    $$\begin{aligned}
        \Phi:H^0(X,\ml O_X(mA))&\rightarrow R/\mfk{a}^k\\
        s_x&\mapsto s_x\opn{mod}\mfk{a}^k
    \end{aligned}$$
    is not zero. Choose a nonzero section $s\in \opn{ker}\Phi$ and let $H:=\opn{div}(s)\in|mA|$. Since $X$ is integral, $H\geq 0$. It's local equation $h$ at $x$ belongs to $\mfk{a}^k$. In particular, $H$ contains $x$ and $\opn{ord}_E(H)=\opn{ord}_E(h)\geq k\opn{ord}_E(\mfk{a})$.\par
    By Remark~\ref{rmk:numerical description of ideal lct}, there exist a prime divisor $E$ over $X$ with $x\in c_X(E)$, such that $a(E,X,0)=c\opn{ord}_E(\mfk{a})$. So for this prime divisor $E$, we have 
    $$\frac{a(E,X,0)}{\opn{ord}_E(H)}\leq \frac{c}{k}.$$\par
    Next, we want to find $t_0$ that satisfies requirement. Take a log resolution $f:Y\rightarrow X$ of $(X,H)$ such that $E$ lives on $Y$. For this resolution, we have $K_Y+\Theta_Y=f^*K_X$, $f^*H=\sum b_iD_i$ and $\Theta_Y=\sum d_iD_i$. Since $X$ is klt, $d_i<1$ for all $i$. Define $I=\{i\,|\,b_i>0,x\in f(D_i)\}$ and $\lambda:=\min\limits_{i\in I}\frac{1-d_i}{b_i}$. Therefore, we have 
    $$0<\lambda\leq \frac{a(E,X,0)}{\opn{ord}_E(H)}\leq \frac{c}{k}.$$\par
    Since $K_Y+\Theta_Y+\lambda f^*H=f^*(K_X+\lambda H)$, the pair $(X,\lambda H)$ is lc near $x$.\par
    There exist an $i_0$ such that $d_{i_0}+\lambda b_{i_0}=1$. Thus $D_{i_0}$ has log discrepancy $0$ for $(X,\lambda H)$ and $x\in c_X(D_{i_0})$. This means $(X,\lambda H)$ is not klt near $x$.\par
    Finally, define $\Delta_0=\lambda H$ and $t_0=\lambda m$. Then $\Delta_0$ is an effective $\bb Q$-Cartier divisor such that $\Delta_0\sim_{\bb Q}\lambda mA=t_0A$. Since $m=k\rho$, we have $$0<t_0=\lambda k\rho\leq c\rho<\beta.$$
    The pair $(X,\Delta_0)$ is lc near $x$ but not klt at $x$ as required.
\end{proof}

Next, we want to cut the LC center on klt variety, before that, we need the following lemma.
\begin{lemma}
\label{lem:adjunction shift}
    Let $X$ be a projective klt variety and $\Theta$ be an effective $\bb Q$-divisor such that $K_X+\Theta$ is $\bb Q$-Cartier and $(X,\Theta)$ is lc near $x\in X$. Let $Z$ be a normal LC center of $(X,\Theta)$ through $x$. Let $G\geq 0$ be a $\bb Q$-Cartier divisor with $Z\not\subset \opn{Supp}G$. Let $\mathbf{B}_{\Theta}$ be the discriminant $b$-divisor of adjunction to $Z$ in~\ref{thm:adjunction and inverse of adjunction}. For every prime divisor $F$ over $Z$, set
    $$ a_Z(F;\Theta):=1-\opn{mult}_F\mathbf{B}_\Theta.$$ Then we have
    $$a_Z(F;\Theta+G)=a_Z(F;\Theta)-\opn{ord}_F(G|_Z).$$
    Moreover, the corresponding discriminant $b$-divisor satisfies
    $$\mathbf{B}_{\Theta+G}=\mathbf{B}_{\Theta}+\overline{G|_Z}.$$
\end{lemma}
\begin{proof}
    Choose a smooth birational model $p:V\rightarrow Z$ such that $F$ appears. For every prime divisor $E$ over $X$ with $c_X(E)=Z$, since $Z\not\subset\opn{Supp G}$, we have $\opn{ord}_E(G)=0$. Hence we have
    $$a(E,X,\Theta+G)=a(E,X,\Theta).$$
    So $E$ is an LC place of $(X,\Theta)$ such that $c_X(E)=Z$ if and only if $E$ is an LC place of $(X,\Theta+G)$ such that $c_X(E)=Z$. Choose such an LC place $E$. We can find a sufficiently high log resolution  $\pi:Y\rightarrow X$ such that $E$ appears and the induced map $g_E:E\rightarrow V$ is a morphism. Write $K_Y+\Theta_Y=\pi^*(K_X+\Theta)$. And we have $K_Y+\Theta_Y+\pi^*G=\pi^*(K_X+\Theta+G)$. Since $E$ has coefficient $1$ in both boundaries, the adjunction give
    $$\Theta_E=(\Theta_Y-E)|_E,$$ and $$\begin{aligned}
        \Theta_E^{+}&=(\Theta_Y+\pi^*G-E)|_E\\
        &=\Theta_E+(\pi^*G)|_E\\
        &=\Theta_E+g^*_Ep^*(G|_Z).
    \end{aligned}$$
    The last equality follows from the compatibility of restriction and pullback for $\bb Q$-Cartier divisors. Set $c=\opn{ord}_F(G|_Z)$. In an neighbourhood of the generic point $\eta_F$, we have $p^*(G|_Z)=cF$. Consequently, for $s\in \bb R$,  we have
    $$\Theta_E^+ +sg^*_E(F)=\Theta_E+(s+c)g^*_E(F).$$
    Define $$\tau_{E,F}(\Theta)=\sup\{s\in \bb R\,|\,(E,\Theta_E+s g^*_E(F))\text{ is sub-lc over }\eta_F\},$$
    and
    $$\tau_{E,F}(\Theta+G)=\sup\{s\in \bb R\,|\,(E,\Theta_E^++s g^*_E(F))\text{ is sub-lc over }\eta_F\}.$$
    Then we have $\tau_{E,F(\Theta+G)}=\tau_{E,F}(\Theta)-c$. By Proposition~\ref{thm:adjunction and inverse of adjunction}(3), we have $a_Z(F;\Theta)=\inf_{E}\tau_{E,F}(\Theta)$. By our construction, the LC place is the same and $c$ is independent of $E$. Therefore, we have
    $$a_Z(F;\Theta+G)=\inf_{E}(\tau_{E,F}(\Theta)-c)=a_Z(F;\Theta)-c.$$
    And since $a_Z(F;\Theta)=1-\opn{mult}_F \mathbf{B}_{\Theta}$, apply the above equality to every prime divisor on $V$, we have $$(\mathbf{B}_{\Theta+G})_V=(\mathbf{B}_\Theta)_V+p^*(G|_Z).$$
    This holds on every smooth model, and hence on every normal model by pushforward. By the definition of Cartier closure, we are done.\par
    In particular, if $V$ satisfies $\mathbf{K}_V+\mathbf{B}_{\Theta}=\overline{K_V+(\mathbf{B}_\Theta)_V}$, then additivity of Cartier closure gives
    $$\mathbf{K}_V+\mathbf{B}_{\Theta+G}=\overline{K_V+(\mathbf{B}_\Theta)|_V}+\overline{G|_Z}=\overline{K_V+(\mathbf{B}_\Theta)|_V+ p^*(G|_Z)}.$$
    Therefore, the perturbed adjoint $b$-divisor descends to the same model $V$.
\end{proof}

\vskip 0.5cm
Now, we establish the LC cutting technique on projective klt variety.
\begin{lemma}
\label{lem:singular-cut}
Let $X$ be a projective klt variety, $A$ be an ample Cartier divisor and $\Delta$ be an effective $\bb Q$-Cartier divisor such that $(X,\Delta)$ is lc near $x$ but not klt at $x$. Let $Z$ be its minimal LC center through $x$. Set $d=\dim Z>0$, $e=\opn{mult}_xZ$ and $b=b_x(X,\Delta)$. Then for every $\varepsilon>0$, there exist an effective $\bb Q$-divisor $D\sim_{\bb Q} A$ and $t\in\bb Q_{>0}$ such that:
\begin{enumerate}
    \item $(X,\Delta+tD)$ is lc near $x$;
    \item the minimal LC center $Z^\p\ni x$ of $(X,\Delta+tD)$ is strictly contained in $Z$, in other words, $Z^\p\subsetneq Z$;
    \item set $b^\p=b_x(X,\Delta+tD)$, then 
    \begin{equation}\label{eq:cut}
 t<(b-b')\left(\frac{e}{A^d\cdot Z}\right)^{1/d}+\varepsilon.
\end{equation} 
\end{enumerate}
\end{lemma}
\begin{proof}
The minimal center $Z$ is normal near $x$. So all the adjunction arguments below are carried out in a neighbourhood of $x$ where $Z$ is normal. And the construction of global section is carried out on the original $X$ and $Z$. \par
\emph{Step 0: Establish the adjunction.} Under the situation of Lemma~\ref{lem:adjunction shift}, we have
\begin{equation}\label{eq:lemma43-adjunction-shift}
 a_Z(F;\Theta+G)
 =a_Z(F;\Theta)-\opn{ord}_F(G|_Z).
\end{equation}
We further claim that
\begin{equation}\label{eq:lemma43-ideal-threshold}
 b_x(X,\Theta)
 =\inf_{c_Z(F)=\{x\}}
   \frac{a_Z(F;\Theta)}{\opn{ord}_F(\mfk m_{Z,x})}.
\end{equation}
\emph{Step 1: Construct the divisor.} Set $\rho:=\left(\frac{A^d\cdot Z}{e}\right)^{1/d}.$ Since $d>0$, we claim that $b>0$. In fact, on a log resolution which principalizes $\mfk m_x$, $b=0$ means that there exist an LC place centered at $x$. So the minimal LC center is $\{x\}$, contradicting with $d>0$. Choose a suitable $q\in\bb Q$ such that $0<q<\rho$ and $b(1/q-1/\rho)<\varepsilon$. For a sufficiently large integer $\ell$, put $k=\lfloor q\ell\rfloor+1$.
Then the Hilbert polynomial gives
$$ h^0\bigl(Z,\ml O_Z(\ell A)\bigr)
 =\frac{A^d\cdot Z}{d!}\ell^d+O(\ell^{d-1}),$$
and the Hilbert-Samuel formula give
$$
 \opn{length}\bigl(\ml O_{Z,x}/\mfk m_{Z,x}^{\,k}\bigr)
 =\frac{e q^d}{d!}\ell^d+O(\ell^{d-1}).
$$
Since $A^d\cdot Z>e q^d$, there exist a nonzero section
$s_Z\in H^0(Z,\ml O_Z(\ell A))$ whose local representative
at $x$ belongs to $\mfk m_{Z,x}^{\,k}$.\par
For sufficiently large $\ell$, Serre vanishing gives the surjection:
$$
 H^0\bigl(X,\ml O_X(\ell A)\bigr)
 \longrightarrow H^0\bigl(Z,\ml O_Z(\ell A)\bigr).
$$
So we can lift $s_Z$ to a section in $\ml O_X(\ell A)$, which defines a divisor $H\in|\ell A|$. Define $D=H/\ell$. Then $D$ is effective, $D\sim A$,
$x\in\opn{Supp} D$, and $Z\not\subset\opn{Supp} D$. Moreover, for every prime divisor $F$ over $Z$ with $c_Z(F)=\{x\}$, we have
\begin{equation}\label{eq:lemma43-jet-order}
 \opn{ord}_F(D|_Z)
 \geq\frac{k}{\ell}\opn{ord}_F(\mfk m_{Z,x})> q\opn{ord}_F(\mfk m_{Z,x}).
\end{equation}

\emph{Step 2: Check all three conditions. }
Let $t:=\opn{lct}_x(X,\Delta;D)$ and let $F\ni x$ be an LC center of $(X,\Delta)$. Since $Z$ is the minial LC center through $x$, we have $F\supset Z$. And since $Z\not\subset \opn{Supp}D$, we have $F\not\subset \opn{Supp}D$. We know that $t>0$. Since $x\in\opn{Supp} D$, the threshold is finite and $t\in\bb Q$. Moreover, $(X,\Delta+tD)$ is lc near $x$ and there exist some prime divisor $E$ such that $x\in c_X(E)$, $\opn{ord}_E(D)>0$ and $a(E,X,\Delta+tD)=0$.\par
Since $Z\not\subset D$, $Z$ is an LC center of $(X,\Delta+tD)$. On the other hand, $W:=c_X(E)\subset\opn{Supp} D$, so $W$
does not contain $Z$. The minimal LC center $Z'\ni x$ of $(X,\Delta+tD)$
is contained in $Z$ and $W$. Hence $Z'\subseteq Z\cap W\subsetneq Z$.\par
Finally, put $b'=b_x(X,\Delta+tD)$. Apply \eqref{eq:lemma43-adjunction-shift},
\eqref{eq:lemma43-ideal-threshold} and\eqref{eq:lemma43-jet-order}, then we have
$$\begin{aligned}
 b^\p &=\inf_{c_Z(F)=\{x\}} \frac{a_Z(F;\Delta)-t\opn{ord}_F(D|_Z)}{\opn{ord}_F(\mfk m_{Z,x})}\\
 &\leq\inf_{c_Z(F)=\{x\}}
   \frac{a_Z(F;\Delta)}{\opn{ord}_F(\mfk m_{Z,x})}
   -t\frac{k}{\ell}\\
 &=b-t\frac{k}{\ell}<b-tq.
\end{aligned}$$
Because $(X,\Delta+tD)$ is lc near $x$, we have $b'\geq0$.
By the construction of $q$, we have
\begin{align*}
 t &<\frac{b-b'}q =\frac{b-b'}\rho+(b-b')\left(\frac1q-\frac1\rho\right)\\
 &\leq\frac{b-b'}\rho
   +b\left(\frac1q-\frac1\rho\right)\\
 &<(b-b')\left(\frac{e}{A^d\cdot Z}\right)^{1/d}
   +\varepsilon.
\end{align*}
It remains to prove the claim~\eqref{eq:lemma43-ideal-threshold}.\par
\emph{Step 3: Proof of the Claim. } Define  
$$\beta:=\inf\limits_{c_Z(F)=\{x\}}\frac{a_Z(F;\Theta)}{\opn{ord}_F(\mfk{m}_{Z,x})}.$$
We shall prove that, for every $\lambda\in \bb Q_{\geq 0}$, $\lambda\leq b_x(X,\Theta)$ if and only if $\lambda\leq \beta$.\par
Shrink $X$ to an affine neighbourhood of $x$ on which $(X,\Theta)$ is lc and $Z$ is normal. Let $f:Y\rightarrow X$ be a log resolution of $(X,\Theta)$ and $\mfk m_x$. By definition, we have $K_Y+\Theta_Y=f^*(K_X+\Theta)$ and $\mfk{m}_x\ml O_Y=\ml O_Y(-P)$ for some effefctive $P$. Moreover, $\opn{Supp}\Theta_Y\cup \opn{Supp}P$ is SNC. The ideal pair $(X,\Theta;\mathfrak m_x^\lambda)$ is lc near $x$ if $a(E,X,\Theta)-\lambda\operatorname{ord}_E(\mathfrak m_x)\geq0$ for every prime divisor $E$ over $X$ whose center contains $x$. So we have 
\begin{equation}
\label{equivalence 1}
    (X,\Theta;\mfk m_x^\lambda) \text{ is lc near } x\,\Longleftrightarrow\, (Y,\Theta_Y+\lambda P)\text{ is sub-lc near } f^{-1}(x).
\end{equation}
Choose a smooth model $p:V\rightarrow Z$ such that $\mathbf{K}_Z+\mathbf{B}_{\Theta}=\overline{K_V+\mathbf{B}_V}$ and $\mfk{m}_{Z,x}\ml O_V=\ml O_V(-Q)$ where $Q\geq 0$ and $\opn{Supp}\mathbf{B}_V\cup \opn{Supp}Q$ is SNC. Here $\mathbf{B}_V$ is the trace of $\mathbf{B}_{\Theta}$ and therefore an ordinary divisor on $V$. The existence of this smooth model is ensured by \cite[Theorem~1.2]{fujino23}, followed by a log resolution principalizing $\mfk{m}_{Z,x}$. Since $(X,\Theta)$ is lc, by \cite[Theorem~1.2(iv)]{fujino23}, we have
$$p(\opn{Supp}(\mathbf{B}_V^{> 1}))=Z\cap \opn{Nlc}(X,\Theta)=\varnothing,$$
so $B_V\leq 1$. Since $\mathbf{K}_Z+\mathbf{B}_{\Theta}=\overline{K_V+\mathbf{B}_V}$, for every prime divisor $F$ over $V$, we have $a(F,V,\mathbf{B}_V)=a_Z(F;\Theta)$. Moreover, since $p$ principalize $\mfk{m}_{Z,x}$, we have $\opn{ord}_F(Q)=\opn{ord}_F(\mfk{m}_{Z,x})$. Therefore, we have 
$$a(F,V,\mathbf{B}_V+\lambda Q)=a_Z(F;\Theta)-\lambda \opn{ord}_F(\mfk{m}_{Z,x}).$$
If $c_X(F)=\{x\}$, then $\opn{ord}_F(\mfk{m}_{Z,x})>0$. For other divisors, add $\lambda Q$ will not change the discrepancy. So we have
\begin{equation}
    \label{equivalence 2}
    (V,\mathbf{B}_V+\lambda Q)\text{ is sub lc near } p^{-1}(x)\,\Longleftrightarrow\,\lambda\leq \beta
\end{equation}
It remains to connect the \eqref{equivalence 1} and $\eqref{equivalence 2}$. Let $u_1,\cdots, u_r$ be the generator of $\mfk{m}_x$. Then their restriction to $Z$ also generates $\mfk{m}_{Z,x}=\mfk{m}_x\ml O_Z$. Fix an integer $N>\lambda$, choose sufficiently general linear combinations $h_j=\sum_{i=1}^r c_{ji}u_i$. Define $H_j:=\opn{div}(h_j)$ and $T_\lambda=\frac{\lambda}{N}\sum_{j=1}^N H_j$.
Let $W$ be the finite dimensional complex vector space spanned by $u_1,\cdots, u_r$. On $Y$, the pullabck $f^*u_i$ generate $\ml O_Y(-P)$. On $V$, their restriction generate $\ml O_V(-Q)$. Thus we have $f^*H_j=P+M_j$ where $M_j$ is the moving part. Therefore, the linear system $\{\opn{div}_Y(f^*h)-P\,|\,0\neq h\in W\}$ is base point free on $Y$. Similarily, we have $p^*(H_j|_Z)=Q+L_j$ where $L_j$ is the moving part and $\{\opn{div}_V(p^*(h|_Z))-Q\,|\, h\in W,\,h|_Z\neq 0\}$ is base point free on $V$.\par
Since $\opn{dim}Z>0$, the ideal $\mfk{m}_{Z,x}\neq 0$. Therefore, the restriction map $W\rightarrow H^0(Z,\ml O_Z)$ is nonzero. It's kernel $W_0=\{h\in W\,|\,h|_Z=0\}$ is a proper linear subspace. Therefore, the tuple $(h_1,\cdots,h_N)\in W^N$ with some $h_j|_Z=0$ form a proper closed subset $$\bigcup_{j=1}^N W^{j-1}\times W_0\times W^{N-j}\subset W^N.$$ So we can choose $(h_1,\cdots,h_N)$ outside this closed subset and then all $h_j|_Z\neq 0$.\par 
  By Bertini's theorem~\ref{thm:bertini}, we can further require the following conditions:
\begin{enumerate}
    \item On $Y$, $M_j$ are reduced, have no common components with one another or with $\opn{Supp}\Theta_Y\cup\opn{Supp}P$ and $\bigcup_{j=1}^N \opn{Supp}M_j\cup \opn{Supp}P\cup \opn{Supp}\Theta_Y$ is SNC .
    \item On $V$, $L_j$ are reduced, have no common components with one another or with $\opn{Supp}Q\cup\opn{Supp}B_V$ and $\bigcup_{j=1}^N \opn{Supp}L_j\cup \opn{Supp}Q\cup \opn{Supp}B_V$ is SNC.
\end{enumerate}
Since each one of these conditions can be translated into an open subset of $W^N$, Bertini's theorem then allows us to find a general member $(h_1,\cdots,h_N)$ that satisfies all these conditions.\par
Now, put $\alpha:=\frac{\lambda}{N}<1$. On $Y$, we have $f^*T_\lambda=\lambda P+\alpha\sum_{j=1}^N M_j$. Therefore, we have 
$$f^*(K_X+\Theta+T_{\lambda})=K_Y+\Theta_Y+\lambda P+\alpha\sum_{j=1}^N M_j.$$\par
By our construction, $\opn{Supp}(\Theta_Y+\lambda P+\alpha\sum_{j=1}^NM_j)$ is SNC. Moreover, $\alpha<1$, so we have
$$\begin{aligned}
    (Y,\Theta_Y+\lambda P+\alpha \sum_j M_j) \text{ is sub-lc near }f^{-1}(x)\\
    \Longleftrightarrow\, (Y,\Theta_Y+\lambda P)\text{ is sub-lc near }f^{-1}(x)
\end{aligned}$$
Together with \eqref{equivalence 1} and Definition~\ref{prop:numerical local discrepancy}, we have
\begin{equation}
\label{equivalence 3}
(X,\Theta+T_{\lambda})\text{ is lc near }x\,\Longleftrightarrow\, \lambda\leq b_x(X,\Theta).
\end{equation}
Similarily, on $V$, we have 
$$\mathbf{B}_V^{\lambda}:=(\mathbf{B}_{\Theta+T_{\lambda}})_V=\mathbf{B}_V+p^*(T_{\lambda}|_Z)=\mathbf{B}_V+\lambda Q+\alpha \sum_{j=1}^N L_j.$$
Therefore
\begin{equation*}
    (V,\mathbf{B}_V^\lambda)\text{ is sub-lc near }p^{-1}(x)\,\Longleftrightarrow\, (V,\mathbf{B}_V+\lambda Q)\text{ is sub-lc near } p^{-1}(x).
\end{equation*}
Together with \eqref{equivalence 2}, this is equivalent to $\lambda \leq \beta$.\par
Since $Z\not\subset T_\lambda$, $Z$ is also an LC center of $(X,\Theta+T_\lambda)$. On $V$, by Lemma~\ref{lem:adjunction shift}, we have 
$$\mathbf{K}_Z+\mathbf{B}_{\Theta+T_\lambda}=\overline{K_V+\mathbf{B}_V+p^*(T_{\lambda}|_Z)}=\overline{K_V+\mathbf{B}_V^{\lambda}}.$$ So $\mathbf{K}_V+\mathbf{B}_{\Theta+T_\lambda}$ descends to $V$. By \cite[Theorem~1.2(iv)]{fujino23}, we have
\begin{equation*}
    x\in\opn{Nlc}(X,\Theta+T_\lambda)\,\Longleftrightarrow\, x\in p(\opn{Supp}((\mathbf{B}_V^{\lambda})^{>1}))
\end{equation*}
Equivalently, we have
\begin{equation}
    \label{equivalence 4}
    (X,\Theta+T_\lambda)\text{ is lc near }x\,\Longleftrightarrow\, (V,\mathbf{B}_V^\lambda) \text{ is sub-lc near }p^{-1}(x).
\end{equation}
Combining with \eqref{equivalence 1} to \eqref{equivalence 4}, we have $\lambda \leq b_x(X,\Theta)$ if and only if $\lambda\leq \beta$ for every rational $\lambda\geq 0$. Thus we have
$$ b_x(X,\Theta)
 =\inf_{c_Z(F)=\{x\}}
   \frac{a_Z(F;\Theta)}{\opn{ord}_F(\mfk m_{Z,x})}.$$
\end{proof}

\begin{lemma}\label{lem:improved-finish}
Let $X$ be a projective klt variety and $A$ an ample Cartier divisor.
Let $\Delta_0\equiv t_0A$ be an effective $\bb Q$-Cartier divisor, with
$t_0>0$. Assume that $(X,\Delta_0)$ is lc near $x$ but not klt
at $x$. Set $d_0=\dim_x\opn{Nklt}(X,\Delta_0)$ and
$$
 \phi(d)=
 \begin{cases}
  0,&d=0,\\
  1,&d=1,\\
  d(d+1)/2-1,&d\geq2.
 \end{cases}
$$
Then for every $\varepsilon>0$, there exist an effective $\bb Q$-Cartier
divisor $\Gamma\equiv tA$ such that
$0<t<t_0+\phi(d_0)+\varepsilon$, $(X,\Gamma)$ is lc near $x$,
and $\{x\}$ is the unique LC center through $x$.
\end{lemma}

\begin{proof}

Let's first consider a minimal LC center $Z$ with $\opn{dim}Z=2$. Set
$b=b_x(X,\Delta)$ and $e=\opn{mult}_xZ$. Han's estimate
\cite[Corollary~4.5]{Han26} tells us that $be\leq2$. Since $A$ is an ample Cartier divisor, we have $A^2\cdot Z\geq1$. By Lemma~\ref{lem:singular-cut}
, we can cut down $Z$ and the increase in the coefficient is $(b-b_1)\sqrt e+\varepsilon$, where $\varepsilon$ is the error term. If the new center is a point, then $b_1=0$. Otherwise it is a normal curve, hence smooth at $x$, and a second cut increase the coefficient by $b_1+\epsilon$. The final coefficient of $A$ is therefore at most
$$
 (b-b_1)\sqrt e+b_1 \leq b\sqrt e \leq\frac2{\sqrt e} \leq2,
$$
apart from the error term $\varepsilon$. \par
If $\opn{dim}Z=1$, in other words, $Z$ is a curve. Then by \ref{lem:singular-cut}, the cutting procedure will increase the coefficient by at most $1$, since its multiplicity is one and $b_x(X,\Delta)\leq1$.\par

If the minimal LC center has dimension $k\geq3$, apply
\cite[Proposition~2.7]{Fujino10} to the pair $(X,0)$,
$D=\Delta$, and $H=(k+\eta)A$.
The numerical hypothesis holds because
$H^k\cdot Z=(k+\eta)^k(A^k\cdot Z)>k^k$.
The dimension of minimal LC center will decrease, and the increase in the coefficient is at most $k+\eta$. Repeat until the dimension of minimal LC center is less or equal to $2$, then return to the surface or curve case.\par
Therefore, for $d_0\geq2$, the total increase in the coefficient, aparts from the error term, is at most
$$
 \sum_{k=3}^{d_0}k+2=\frac{d_0(d_0+1)}2-1.
$$
Once $\{x\}$ becomes an LC center, we can perturb it and make it the unique LC center through $x$. Morr precisely, by \cite[Lemma~2.8]{GL24}, we have
$(1-\delta)\Delta+\delta E_\delta$, with $E_\delta\sim aA$ for fixed $a>0$ and arbitrarily small $\delta>0$. And $\{x\}\ni x$ is the unique LC center of $(X,(1-\delta)\Delta+\delta E_\delta)$. This perturbation changes the final coefficient of $A$ sufficiently small so that it can be absorbed by the error term $\varepsilon$.
\end{proof}

\begin{proposition}\label{prop:improved-singular}
Let $X$ be a projective klt variety of dimension $n\geq2$. Let
$A$ be an ample Cartier divisor on $X$. Set $V=A^n$. Define
\begin{equation}\label{eq:new-R}
 R_n(V):=
 \begin{cases}
  \frac n{V^{1/n}}+1,&n=2,\\
  \frac n{V^{1/n}}+\frac{n(n-1)}{2}-1,&n\geq3.
 \end{cases}
\end{equation}
For every singular point $x\in X$, there exist an effective
$\bb Q$-Cartier divisor $\Gamma_x\equiv t_xA$, with
$0<t_x<R_n(V)$, such that $(X,\Gamma_x)$ is lc near $x$ and
$x$ is an isolated non-klt point.
\end{proposition}

\begin{proof}
    Let $\beta:=n/V^{1/n}$. By Theorem~\ref{thm:bdd normal vol}, we have 
    $\widehat{\opn{vol}}(x,X)<n^n=\beta^nV$. By Proposition~\ref{prop:initial singular pair}, there exist 
$\Delta_0\sim t_0A$, with $0<t_0<\beta$ such that $(X,\Delta_0)$ is lc near
$x$ and not klt at $x$ and $d_0\leq n-1$. Choose $0<\varepsilon<\beta-t_0$ and apply
Lemma~\ref{lem:improved-finish}. Since $\phi$ is increasing, we have
$$
 0<t_x<t_0+\phi(d_0)+\varepsilon
       <\beta+\phi(n-1)=R_n(V).
$$
And the unique LC center through $x$ is $\{x\}$, so after shrinking
, the non-klt locus is exactly $\{x\}$.
\end{proof}

%% file: MainTheoremandApplication.tex
\section{Main Theorem and Application}

\begin{lemma}
\label{lem:smooth Bs locus in Exc}
    Let $X$ be a smooth minimal projective variety of general type and $f:X\rightarrow Y$ be the canonical contraction in Prop~\ref{prop:canonical contraction}. If $m\geq 2n+1$, then 
    $$\opn{Bs}|mK_X|\subset \opn{Exc}(f).$$
\end{lemma}
\begin{proof}
    Apply Theorem~\cite[Theorem~5.4]{Han26} with $L=K_X$, we then have
    $$\opn{Bs}|mK_X|=\opn{Bs}|K_X+(m-1)K_X|\subset \mathbf{B}_+(K_X)$$
    Since $K_X=f^*K_Y$ and $K_Y$ is ample, by \cite[Lemma~3.1]{BCL14}, we have
    $$\mathbf{B}_+(K_X)=f^{-1}(\mathbf{B}_+(K_Y))\cup\opn{Exc}(f)=\opn{Exc}(f).$$
\end{proof}

\vskip 0.5cm

\begin{proof}[Proof of Theorem~\ref{maintheorem3}]
    Consider the canonical contraction $f:X\rightarrow Y$ in Proposition \ref{prop:canonical contraction}. By Lemma~\ref{lem:smooth Bs locus in Exc}, we have $\opn{Bs}|mK_X|\subset \opn{Exc}(f)$.\par
    Next, we handle with the singular point case. For every $y\in f(\opn{Exc}(f))\subset\opn{Sing}Y$, Proposition~\ref{prop:improved-singular} provides an effective $\bb Q$-Cartier divisor $\Gamma_y\equiv t_yK_Y$ with $t_y<R_n(V)\leq m-1$ such that $(Y,\Gamma_y)$ is lc near $y$ and $y$ is an isolated non-klt point. Note that $mK_Y-(K_Y+\Gamma_y)\equiv (m-1-t_y)K_Y$ is ample. So by Proposition~\ref{prop:isolated-nadel}, there exist a global section $s$ of $mK_Y$ nonzero at $y$. Then $f^*s$ nonvanishes on the whole fiber $f^{-1}(y)$. Therefore, we have $\opn{Bs}|mK_X|\cap \opn{Exc}(f)=\varnothing$. We conclude that $\opn{Bs}|mK_X|=\varnothing$.
\end{proof}

\vskip 0.5cm
\begin{proof}[Proof of Theorem~\ref{maintheorem}]
    Since $X$ is a smooth minimal projective variety of dimension $n\geq 4$ of general type, $K_X$ is a nef and big Cartier divisor and $K_X^n\geq 1$. So $R_n(K_X^n)\leq \frac{n(n+1)}{2}-1$ and $R_n(K_X^n)\geq 2n$ when $n\geq 3$.
\end{proof}
\vskip 0.5cm

\begin{corollary}
    Let $X$ be a smooth minimal projective $4$-fold of general type, then $|10K_X|$ is base point free. If $K_X^4\geq 4$, then $|9K_X|$ is base point free. 
\end{corollary}
\begin{proof}
    Direct calculation shows that $2n+1=9$ and $R_4(K_X^4)+1\leq 10$ since $K_X^4\geq 1$. If $K_X^4\geq 4$, then $R_4+1\leq 9$.
\end{proof}